\documentclass{amsart}

\usepackage{amsmath, amsfonts, amssymb, amsthm}

\newtheorem{theorem}{Theorem}[section]
\newtheorem{lemma}[theorem]{Lemma}

\newtheorem{proposition}[theorem]{Proposition}

\theoremstyle{definition}

\newtheorem{remark}[theorem]{Remark}

\numberwithin{equation}{section}

\begin{document}

\title[Lipschitz conditions on bounded harmonic functions]{Lipschitz conditions on bounded harmonic functions on the upper half-space}

\author{Marijan Markovi\'{c}}

\begin{abstract}
This work is devoted to     Lipschitz conditions on bounded harmonic functions on the upper half-space in $\mathbb {R}^n$. Among other results we prove the following  one. Let $U(x',x_n)$   be
a real-valued  bounded  harmonic function  on the upper half-space $\mathbb {R}^n_+ = \{(x',x_n):x'\in \mathbb{R}^{n-1}, x_n\in (0,\infty)\}$, which is continuous on the closure of this domain.
Assume that for $\alpha\in (0,1)$ there exists a constant $C$ such that for every  $x'\in \mathbb{R}^{n-1}$  we have  $| |U|(x',x_n)  -  |U|(x',0)|\le Cx_n^\alpha,\,   x_n\in (0,\infty)$. Then
there   exists  a constant $\tilde {C}$  such  that $|U(x) - U (y)| \le  \tilde{C}  |x-y|^\alpha,\, x,y\in \mathbb{R}^{n}_+$.
\end{abstract}

\address{Faculty of Science and Mathematics\endgraf
University of Montenegro\endgraf
D\v{z}ord\v{z}a Va\v{s}ingtona BB\endgraf
81000 Podgorica
\endgraf Montenegro}

\email{marijanmmarkovic@gmail.com}



\maketitle

\section{Introduction and the main results}

In this paper $\alpha$ is always a real number in $(0,1)$. We denote by $\Lambda^\alpha(E)$ the Lipschitz class of (complex-valued) functions on a set $E\subseteq \mathbb{R}^n$. A function $f$
belongs to this  class  if there exists  a  constant $C = C_f$ such that
\begin{equation*}
|f(x)  -  f (y)| \le  C  |x-y|^{\alpha},\quad x,   y\in E.
\end{equation*}

We say that two non-negative quantities (possibly with infinite values somewhere)  $A(f)$ and  $B(f)$,  defined (for example) on a class of functions   $\mathcal {F}$,  are equivalent if there
exist positive constants  $C_1$ and $C_2$ such that  $ A(f)\le C_1  B(f)$ and   $B(f)\le C_2 A(f)$ for every  $ f\in \mathcal{F}$.

On the class of  analytic functions on the unit disk  $\mathbb{D}$ in  $\mathbb{C}$, continuous on the closed unit disk, introduce  the following  quantities:
\begin{equation*}\begin{split}
&N_1(f) = \sup _{x, y\in \overline{\mathbb{D}},\,  x\ne y} \frac {|f(x) - f(y)|}{|x-y|^\alpha};\quad
N_2(f)= \sup _{x, y\in \overline{\mathbb{D}},\,  x\ne y} \frac {||f|(x) - |f|(y)|}{|x-y|^\alpha};
\\&N_3(f) = \sup _{\zeta, \eta\in \partial\mathbb{D},\, \zeta\ne\eta} \frac {||f|(\zeta) - |f|(\eta)|}{|x-y|^\alpha}
+\sup_{z\in\mathbb{D}} \frac {\mathrm {P}[|f|](z)  - |f|(z) }{(1-|z|)^\alpha};
\\&N_4(f)= \sup _{\zeta, \eta\in \partial\mathbb{D},\, \zeta\ne \eta} \frac {||f|(\zeta) - |f|(\eta)|}{|\zeta-\eta|^\alpha}
+\sup_{r\in(0,1),\, \zeta\in\partial\mathbb{D}}\frac {||f|(\zeta)- |f|(r\zeta) |}{(1-r)^\alpha}.
\end{split}\end{equation*}
Here $\mathrm {P} [\varphi]$ denotes the Poisson  integral of a continuous  function $\varphi$ on the unit circle.

In 1997, Dykonov \cite{DYAKONOV.ACTM} proved  that $N_j$, $j=1,2,3,4$,  are mutually  equivalent. As a consequence, an analytic function $f$ on the unit disk $\mathbb{D}$,        continuous on
$\overline{ \mathbb {D}}$, belongs to the Lipschitz class  $\Lambda^\alpha (\overline {\mathbb {D}})$ if and only if  its modulus $|f|$ belongs to  $\Lambda^\alpha (\partial \mathbb {D})$  and
satisfies a Lipschitz-type condition along  every  radius of the unit disk. It also  follows an interesting result that the function $f$  belongs  to $\Lambda^\alpha(\overline{\mathbb{D}})$ if
and only if its  modulus belongs to  the  same class. These results have various  extensions.              We refer to  \cite{DYAKONOV.AIM, DYAKONOV.MZ.2005, DYAKONOV.MZ.2006} for some of them.
Pavlovi\'{c} \cite{PAVLOVIC.ACTM}   gave a  new  proof  of the Dyakonov theorem.

In  \cite{PAVLOVIC.RMI} Pavlovi\'{c} considered Lipschitz  conditions  on  a real-valued   harmonic function on the unit ball $\mathbb{B}^n$ in $\mathbb{R}^n$.        He showed that a such one
function belongs  to $\Lambda^\alpha (\mathbb {B}^n)$  if and only  if  its  modulus satisfies a Lipschitz-type  condition over every radius of the unit ball. This result is formulated  in the
following  proposition.

\begin{proposition}[Pavlovi\'{c}]
For  a real-valued harmonic function $U$ on  $\mathbb{B}^n$, continuous on the closed unit ball,  the  following  three     conditions are equivalent:
\begin{equation*}\begin{split}
  &\text{(B1) $\exists C_1$ such that } |U(x) - U(y)| \le C_1 |x-y|^\alpha,  x,y\in \mathbb{B}^n;
\\&\text{(B2) $ \exists  C_2$ such that }| |U| (\zeta) - |U| (\eta)| \le C_2  |\zeta-\eta|^\alpha, \zeta,\eta \in \partial \mathbb {B}^n ;
\\&\text{(B3) $\exists  C_3$ such that }| |U| (\zeta) - |U| (r \zeta) |  \le C_3  (1-r)^\alpha,  r\in (0,1),\zeta \in \partial \mathbb {B}^n.
\end{split}\end{equation*}
\end{proposition}

It  is  natural to ask  what can be achieved  for  harmonic function on more general   domains in $\mathbb{R}^n$. The case of bounded domains with smooth boundaries is considered by Ravisankar
\cite{RAVISANKAR.PHD, RAVISANKAR.CVEE}. In these works the  author considers the Lipschitz-type behaviour of a harmonic function  along  curves that satisfy     the  transversal condition with
respect to the boundary  of the domain.

One of our aims in this article  is to  find an analog of the  Pavlovi\'{c} theorem for harmonic functions on  the upper  half-space in $\mathbb {R}^n$. This is the content of the next theorem.

\begin{theorem}\label{TH.MAIN.4}
For a bounded  real-valued harmonic functions  $U$          on the upper  half-space $\mathbb{R}^n_+$, continuous on the closure of the domain, the following three statements  are   equivalent:
\begin{equation*}\label{H123}\begin{split}
  &\text{(H1)   $\exists  C_1$ such that } |U(x) - U(y)| \le C_1 |x-y|^\alpha,  x,y\in \mathbb {R}^n_+;
\\&\text{(H2)   $\exists C_2$ such that }  | |U| (x',0) - |U| (y',0)| \le C_2  |x'- y'|^\alpha, x', y' \in \mathbb {R}^{n-1};
\\&\text{(H3)   $\exists C_3$ such that }  | |U| (x',x_n) - |U| (x',0 ) |  \le C_3  x_n^\alpha,  x' \in  \mathbb {R}^{n-1}, x_n\in (0,\infty).
\end{split}\end{equation*}
\end{theorem}

Beside the above theorem, our  main results are collected in the next three theorems.

\begin{theorem}\label{TH.MAIN.1}
Let  $U$ be a real-valued  bounded harmonic function on  $\mathbb{R}^n_+$, continuous on $\overline{\mathbb{R}^n_+}$, and  assume that the function  $U(t,0)$, $t\in \mathbb{R}^{n-1}$ satisfies
\begin{equation*}
| U(x', 0 ) - U (y',0) | \le C |x'-y'|^\alpha, \quad x',y'\in \mathbb {R}^{n-1},
\end{equation*}
where $C$ is a constant.   Then  we have
\begin{equation*}
\left|\frac {\partial}{\partial x_n}U(x',x_n) \right| \le \frac {4 (n-1)^{\frac 12} }{1-\alpha} C x_n^{\alpha -1},
\end{equation*}
and
\begin{equation*}
\left|\nabla U (x',x_n)\right| \le  \frac{16 (n+1)}{(1-\alpha) ^2}C x_n^{\alpha-1},\quad   (x',x_n)\in\mathbb{R}^n_+.
\end{equation*}
The function $U$  belongs to the Lipschitz class  $\Lambda ^ \alpha (\mathbb{R}^n_+)$,  and
\begin{equation*}
|U(x) - U(y)| \le  \frac{64 (n+1)}{\alpha (1-\alpha) ^2}C   |x-y|^{\alpha}, \quad x,y\in\mathbb{R}^n_+.
\end{equation*}
\end{theorem}

\begin{theorem}\label{TH.MAIN.2}
Let $U$  be a bounded real-valued harmonic function on $\mathbb{R}^n_+$,   continuous on the closure of this domain, and let it satisfies the following Lipschitz-type condition: There exists a
constant   $C$  such that
\begin{equation*}
|U(x', x_n) - U (x', y_n)| \le   C   y_n^\alpha,\quad  x'\in \mathbb{R}^{n-1},\, 0< x_n < y_n.
\end{equation*}
Then  we have
\begin{equation*}
\left|\frac {\partial}{\partial x_n}U(x', x_n)\right|\le   13 (n+2)^{\frac 12} C x_n^{\alpha-1},
\end{equation*}
and
\begin{equation*}
|\nabla U (x',x_n)| \le  \frac  { 52 (n+2)} {1-\alpha} C    x_n ^{\alpha-1} ,\quad x= (x',x_n) \in \mathbb{R}^n_+.
\end{equation*}
The function $U$ belongs to the Lipschitz class  $\Lambda ^ \alpha (\mathbb{R}^n_+)$,  and
\begin{equation*}
|U(x) - U (y)|\le \frac{208 (n+2)} {\alpha (1-\alpha)}  C  |x-y|^{\alpha},\quad x,y\in\mathbb {R}^n_+.
\end{equation*}
\end{theorem}

\begin{theorem}\label{TH.MAIN}
On  the class of vector-valued bounded harmonic functions on  $\mathbb{R}^n_+$,  continuous on the closure of the domain,  the following three semi-norms  are  equivalent:
\begin{equation*}\begin{split}
&\|U\|_1  =  \sup _{x,  y\in \mathbb{R} ^n_+,\, x\ne y  } \frac {|U(x) - U(y)|}{|x-y|^\alpha};\quad
\|U\|_2  = \sup _{ x',  y'\in \mathbb{R} ^{n-1}, x'\ne y' } \frac {| U(x',0) - U(y',0)|}{|x'-y'|^\alpha};
\\&\|U\|_3 = \sup_{x'\in \mathbb{R}^{n-1}, x_n, y_n\in (0,\infty),\, x_n\ne y_n} \frac { |  U (x',x_n) - U  (x',y_n) |  } {|x_n- y_n|^\alpha}.
\end{split}\end{equation*}
\end{theorem}

In connection with Theorem \ref{TH.MAIN.1}  we refer to \cite[Lemma 4, p. 143]{STEIN.BOOK} and the text after it. We will prove  here this result  by carefully looking at constants.   After we
estimate the partial derivative of a harmonic function  in  the  last  variable, we  derive the gradient estimate using the Schwarz lemma for harmonic functions. The similar approach we use in
proving Theorem \ref{TH.MAIN.2} where harmonic functions on the upper half-space satisfy the Lipschitz-type condition along half-lines orthogonal to the boundary of the domain.

The unit ball case of Theorem \ref{TH.MAIN.1} and Theorem \ref{TH.MAIN.2} (without precise constants) may also be found in the Pavlovi\'{c}   work \cite{PAVLOVIC.RMI}. Although, even for $n>2$,
there is a connection via the  Kelvin transform between harmonic functions on the unit ball and the upper half-space  (we refer to the  seventh chapter in \cite{AXLER.BOOK}), it seems that one
cannot derive above theorems  from the unit ball case.

\section{Proofs of the main theorems}

Before we start  proving  our theorems we will collect  needed  auxiliary results.

The  problem of finding  the best pointwise  gradient  estimate in the class  of bounded     harmonic functions on the  unit  ball or the upper  half-space in  $\mathbb{R}^n$  was known as the
Khavinson problem. The Khavinson problem in  the upper half-space setting  is solved by Kresin and Maz'ya  \cite{KRESIN.DCDS}. Their result is given in the  proposition stated below. It   says
that  the optimal estimate of the  $n$-th partial derivative coincides  with the optimal estimate  of  the gradient.

\begin{proposition}[Kresin-Maz'ya]\label{PROP.KHAVINSON}
If  $U$  is  among  bounded   harmonic functions on the upper half-space   $\mathbb {R}^n_+$,                           then  we have    the following  pointwise optimal    gradient   estimate
\begin{equation*}
\left|\nabla U(x',x_n)\right|, \left|\frac {\partial}{\partial x_n} U (x', x_n)\right|\le      M_n  x_n  ^{-1}\sup _{y\in  \mathbb{R}^n_+} |U(y)|, \quad (x',x_n)\in \mathbb{R}^n_+,
\end{equation*}
where
\begin{equation*}\begin{split}
M_n  & =  \frac{4 (n-1)^{\frac{n-1}2}}{n^\frac {n}{2}} \frac{m_{n-1}(\mathbb{B}^{n-1})}{m_{n}(\mathbb{B}^{n})},
\end{split}\end{equation*}
and  $m_n(\mathbb {B}^n)$ is the volume of  the  unit ball  $\mathbb{B}^n$.
\end{proposition}

The Khavinson problem   for the unit ball was considered by various authors. We refer to the Liu paper \cite{LIU.MA} where the  problem is completely solved and the main theorem says  that the
optimal  estimate for the gradient is the same as the  optimal estimate in the  radial direction. The   derivative of a function $U$ in a direction $l\in\mathbb{R}^n$, $|l|=1$,  is denoted  by
\begin{equation*}
\frac{\partial}{\partial \ell}U(x) = \lim _{t\to 0} \frac{U (x+t\ell ) - U(x)}{t} = \left<\nabla U(x),\ell\right>.
\end{equation*}

\begin{proposition}[Liu]
If  $U$  is among real-valued  bounded   harmonic functions on the unit  ball    $\mathbb {B}^n$,                    then  we have    the following       pointwise optimal  gradient   estimate
\begin{equation*}
\left|\nabla U(x)\right|, \left|\frac {\partial}{\partial r(x)}U (x)\right|\le N_n (x)  (1-|x|)  ^{-1}\sup _{y\in  \mathbb{B}^n} |U(y)|, \quad x\in \mathbb{B}^n,
\end{equation*}
where $r(x) = x/|x|$, if $x\ne 0$, and  arbitrary direction if $x=0$, and
\begin{equation*}\begin{split}
N_n(x)  & =  \frac{m_{n-1}(\mathbb{B}^{n-1})}{m_{n}(\mathbb{B}^{n})}\frac {n-1}{|x|+1} 
 \int_{-1}^1\frac{|t-\frac{n-2}{n}|x||(1-t^2)^{\frac{n}2-\frac 32}}{(1-2t|x|+|x|^2)^{\frac n2-1}}dt.
\end{split}\end{equation*}
\end{proposition}

The next  proposition may be found in \cite[Theorem 6.26]{AXLER.BOOK}. It is a  special case of the preceding one for $x = 0$.

\begin{proposition}[the Schwarz lemma for harmonic functions]\label{PROP.AXLER}
Let  $U$ be   a  real-valued   harmonic function on  $\mathbb {B}^n$ which satisfies   $\sup_{x\in \mathbb{B}^n}|U(x)| \le 1$.  Then
\begin{equation*}
|\nabla U (0)| \le  \frac {2 m_{n-1}(\mathbb {B}^{n-1})}{ m_ n(\mathbb {B}^n )}.
\end{equation*}
\end{proposition}

\begin{lemma}\label{LE.KN}
Let $U$ be a real-valued  bounded  harmonic  function on  a  domain  $D\not = \mathbb{R}^n$,   and  $x\in D$. Then   we have
\begin{equation*}
|\nabla U (x) |, \left|\frac {\partial}{\partial \ell}U(x)\right| \le  \frac{K_n} {d(x,\partial D)} \sup_{y\in D} |U(y)|,\quad l\in\mathbb{R}^n, |l|=1,
\end{equation*}
where $d(x,\partial D)$ is the distance function, and
\begin{equation*}\begin{split}
K_n &= \frac {2m_{n-1} (\mathbb{B}_{n-1})} {m_n (\mathbb{B} _n)}.
\end{split}\end{equation*}
\end{lemma}

\begin{proof}
Since
\begin{equation*}
\left|\nabla U(x)\right| = \sup_ {\ell\in \mathbb{R}^n,\,  |\ell|=1}  \left| \frac{\partial}{\partial \ell}U(x) \right|,
\end{equation*}
it is enough to  prove the gradient  estimate.

Assume that $U$ is not  a  constant  function. Denote   $M  = \sup_{y\in D} | U(y)| $, and let us  consider 
\begin{equation*}
V (z) =  M^{-1} U (x  +  d(x,\partial D)  z),\quad   z\in \mathbb{B}^n,
\end{equation*}
which is a harmonic function  on  the unit ball  $\mathbb{B}^n$, and we have $\sup_{z\in \mathbb{B}^n}|V (z)|\le1$.    Therefore,  we may apply Proposition \ref{PROP.AXLER} on the function $V$.
Since
\begin{equation*}
\nabla V (0) = M^{-1} d(x,\partial D)  \nabla U(x),
\end{equation*}
and since  by the proposition  we have  $|\nabla V (0)|\le K_n$,  we deduce    the gradient  estimate for  $U$.
\end{proof}

\begin{remark}
In the sequel we will several times  use  the  estimate  $K_n\le (n+2)^{\frac 12}$,  which follows from the  Gautschi inequality for the  Gamma function
\begin{equation*}
x^{1-s} \le\frac{\Gamma(x  + 1)}{\Gamma(x+s) }\le (x+1)^{1-s},\quad x\in (0,\infty),\, s\in (0,1).
\end{equation*}
Indeed, since $m_n (\mathbb {B}^n ) = \frac {\pi ^{\frac n2}}{\Gamma (\frac n2+1)}$, we have
\begin{equation*}\begin{split}
K_n & =   \frac{2}{\pi^{\frac 12}} \frac{\Gamma(\frac n2 +1)}{\Gamma(\frac n2+\frac {1}2)} 
\le  \frac{2}{\pi ^{\frac 12}} \left(\frac n2+1\right)^{\frac 12}
\le (n+2)^{\frac 12}.
\end{split}\end{equation*}
\end{remark}

\begin{lemma}\label{LE.GRAD.ALPHA}
If a bounded harmonic  function  $U$ on   $\mathbb {R}^n_+ $ satisfies
\begin{equation*}
\left|\frac {\partial}{\partial x_n}U (x', t)\right|\le   C t^{\alpha-1},\quad x = (x', t)\in \mathbb {R}^n_+,
\end{equation*}
then we  have
\begin{equation*}
|\nabla U (x',x_n)| \le  \frac {4 (n+2)^\frac12}  {1-\alpha} C x_n ^{\alpha-1}, \quad x = (x', x_n)\in \mathbb {R}^n_+.
\end{equation*}
\end{lemma}

\begin{proof}
Applying  Lemma \ref{LE.KN}  on the harmonic function  $ \frac {\partial U}{\partial x_n}$  on the open ball   $D=B((x',t),\frac t2)$ (where it is certainly bounded), and   having on mind  the
assumed estimate  of the $n$-th partial  derivative, we obtain
\begin{equation*}\begin{split}
\left|\frac {\partial^2} {\partial \ell \partial x_n} U (x', t) \right|
&\le\frac {K_n}{d( (x', t) ,\partial D)} \sup_{y = (y',y_n)\in D}\left|\frac{\partial}{\partial x_n}U(y)\right|
\\&\le \frac{2K_n} {t}   C y_n^{\alpha-1}= 2^{2-\alpha} K_n C t^{\alpha-2},
\end{split}\end{equation*}
because  for $y = (y',y_n)\in {B}((x',t),\frac {t}2)  $ we have $y_n\ge\frac { t}2$.

After permuting  the derivatives, we have
\begin{equation*}
\left|\frac {\partial^2} {\partial x_n\partial \ell}U (x',t)  \right| \le  4 K_n  C t^{\alpha-2},\quad (x',t)\in \mathbb {R}^n_+.
\end{equation*}

Let $x= (x',x_n) \in \mathbb{R}^n_+$. Integrating over  $[x_n, s]$,  we  obtain the estimate of the partial derivative  in arbitrary  direction $\ell$, $|\ell|=1$:
\begin{equation*}\begin{split}
\left|\frac {\partial} {\partial {\ell }} U (x',s)-\frac {\partial} {\partial {\ell }} U (x',x_n) \right|
&=\left|\int_{x_n}^{s}  \frac {\partial^2} {\partial x_n\partial \ell}U (x',t) dt \right|
\\&\le  4K_n  C\int_{x_n}^{s} t^{\alpha-2}dt
\\&=\frac {4 K_n  C}  {1-\alpha} ( x_n^{\alpha-1} - s^{\alpha-1} ).
\end{split}\end{equation*}
Since the gradient of $U$ at $(x',s)$ converges to zero as $s\to \infty$, which may be seen from Proposition \ref{PROP.KHAVINSON}, if we let $s\to \infty$ in the above inequality,    we obtain
\begin{equation*}
\left|\frac{\partial}{\partial \ell}U(x',x_n)\right|\le\frac {4 K_n  C}  {1-\alpha} x_n^{\alpha-1},\quad x = (x',x_n)\in \mathbb {R}^n_+.
\end{equation*}
Since the above estimate holds for every direction  $\ell$,  we derive the gradient estimate
\begin{equation*}
|\nabla U (x)| 
= \sup_{\ell\in\mathbb{R}^n,\, |\ell|=1}
\left|\frac{\partial}{\partial \ell}U(x)\right| \le \frac {4 K_n C}  {1-\alpha} x_n^{\alpha-1},\quad x = (x',x_n)\in \mathbb {R}^n_+,
\end{equation*}
from which follows  the  estimate we aimed  to prove, since $K_n\le (n+2)^{\frac 12}$.
\end{proof}

\begin{lemma}\label{LE.GM}
Assume that    $f\in C^1(\mathbb {R}^n_+ )$   satisfies
\begin{equation}\label{EQ.F.GRAD}
|\nabla f(z',z_n)|\le C z_n^{\alpha-1},\quad  z = (z',z_n) \in \mathbb{R}^n_+.
\end{equation}
Then  $f$  belongs to the Lipschitz  class $\Lambda^\alpha (\mathbb {R}^n_+)$, and we have
\begin{equation*}
|f(x) - f(y)| \le   \frac{4}\alpha  C |x-y|^{\alpha}, \quad x, y\in \mathbb{R}  ^n_+.
\end{equation*}
In particular $f$ has continuous extension on the closed upper half-space. 
\end{lemma}

\begin{proof}
Let $x, y\in \mathbb{R}^n_+$. It is possible to find an open ball   $B(z,r)\subseteq \mathbb{R}^n_+$   which contains $x$ and $y$. Let $\gamma$ be a part of a circle which is orthogonal to the
sphere  $\partial B (z,r)$  with endpoints at $x$ and $y$.  As in Gehring and Martio work \cite[p. 205]{GEHRING.AASF}, one derives  the inequality
\begin{equation*}
\int_\gamma d(z,\partial \mathbb{R}^n_+) ^{\alpha-1} |dz|\le \frac{\pi}{\alpha 2^\alpha}|x-y|^\alpha
\end{equation*}
(which actually shows that the upper half-space is a  $\mathrm {Lip}_\alpha$-extension domain; for more of this type of domains we refer to  \cite{LAPP.AASF}).

Since the inequality \eqref{EQ.F.GRAD}  may be rewritten in the form
\begin{equation*}
|\nabla f (z',z_n)| \le  C  d(z,\partial \mathbb {R}^n_+) ^{\alpha-1},
\end{equation*}
it follows
\begin{equation*}\begin{split}
|f(x) - f(y)| & \le \int_\gamma |\nabla f (z)||dz| \le C  \int_\gamma d(z,\partial \mathbb{R}^n_+)^{\alpha-1}|dz|\le \frac{4}{\alpha} C |x-y|^{\alpha},
\end{split}\end{equation*}
which we aimed to prove.
\end{proof}

\begin{proof}[Proof of Theorem \ref{TH.MAIN.1}]
Let   $U$  be  a real-valued  bounded harmonic function on  $\mathbb{R}^n_+$, continuous on the closed upper half-space. The theorem assumes  that  $U(t ,0)$, $t\in\mathbb{R}^{n-1}$, satisfies
\begin{equation}\label{EQ.TH.MAIN.1}
| U(x', 0 ) - U (y',0) | \le C |x'-y'|^\alpha, \quad x',y'\in \mathbb {R}^{n-1},
\end{equation}
where $C$ is a constant.

It is enough to  prove the first estimate in this theorem concerning the $n$-th  partial derivative of $U$.   To obtain the gradient estimate, we have only to apply   Lemma \ref{LE.GRAD.ALPHA}.
Then the third  estimate follows from Lemma \ref{LE.GM}.

For  the theory of harmonic functions on  the  unit ball  or the upper half-space  in  $\mathbb{R}^n$ we refer to \cite{AXLER.BOOK} and  \cite{STEIN.BOOK}.  However, it is well known fact that
the function  $U (x',x_n)$   may be represented as the Poisson integral of the bounded  continuous  function  $U(t,0)$, $t\in \mathbb{R}^{n-1}$, i.e.,
\begin{equation*}
U(x',x_n) = \int_{\mathbb{R}^{n-1}} P ((x',x_n), t) U(t, 0) dm_{n-1}(t),\quad (x',x_n)\in \mathbb{R}^n_+.
\end{equation*}
where   $P((x',x_n), t)$   is the Poisson kernel for the upper half-space, and $m_{n-1}$ is the Lebesgue measure on $\mathbb {R}^{n-1}$. Recall that the expression for the   Poisson kernel for
the upper half-space  $\mathbb{R}^n_+$   is given by
\begin{equation*}\begin{split}
P((x',x_n), t) & =  \frac {\Gamma (\frac n2)}{\pi ^{\frac n2}} \frac{x_n}{ ( |x' - t|^2 + x_n^2) ^{\frac n2}},\quad x', t\in \mathbb{R}^{n-1},\, x_n\in (0,\infty).
\end{split}\end{equation*}

It is straightforward  to  calculate
\begin{equation*}\begin{split}
\frac {\partial}{\partial x_n} P((x',x_n),t)
= \frac {\Gamma (\frac n2)}{\pi ^{\frac n2}} \frac{|x'-t|^2 - (n-1)x_n^2}{ ( |x'-t|^2 + x_n^2) ^{\frac n2+1}}.
\end{split}\end{equation*}
Since for every  $(x',x_n)\in\mathbb{R}^n_+$ we have
\begin{equation*}
\int_{\mathbb{R}^{n-1}} \frac {\partial}{\partial x_n} P((x',x_n),t)dm_{n-1}(t) = 0
\end{equation*}
(which  may be checked  directly, or  by the  differentiation a constant function on  $\mathbb{R}^n_+$),  we obtain
\begin{equation*}\begin{split}
\frac {\partial}{\partial x_n}U (x',x_n)  & = \int_{\mathbb {R}^{n-1}}\frac {\partial} {\partial x_n} P((x', x_n),t) U (t, 0)dm_{n-1}(t)
\\& = \int_{\mathbb {R}^{n-1}}\frac {\partial} {\partial x_n} P((x',  x_n), t) ( U (t, 0) - U (x',0))dm_{n-1}(t).
\end{split}\end{equation*}
Introducing   $ u = x'- t $ it follows
\begin{equation*}\begin{split}
\left|\frac {\partial}{\partial x_n} U (x',x_n)\right| & 
\le \int_{\mathbb {R}^{n-1}} \left|\frac {\partial} {\partial x_n} P( (x', x_n),x' -u )\right| | U (x'-u, 0) - U (x',0)|dm_{n-1}(u).
\end{split}\end{equation*}
Having on mind that $U(t,0)$  satisfies \eqref{EQ.TH.MAIN.1},   we  obtain
\begin{equation*}\begin{split}
\left|\frac {\partial}{\partial x_n} U (x',x_n)\right| &
\le C \int_{\mathbb {R}^{n-1}} \left|\frac {\partial} {\partial x_n} P((x',x_n),x'- u )\right| | u |^{\alpha}dm_{n-1}(u).
\end{split}\end{equation*}

Our next  aim is to find  a  good  estimate for  the integral on the right side. First of all, we  have
\begin{equation*}\begin{split} 
\int_{\mathbb {R}^{n-1}} \left|\frac {\partial} {\partial x_n} P((x',x_n),x'- u )\right|& | u |^{\alpha}dm_{n-1}(u)
\\&\le  \frac {\Gamma (\frac n2)}{\pi ^{\frac n2}}\int_{\mathbb {R}^{n-1}} \frac{|u|^2 + (n-1) x_n^2}{ ( |u|^2 + x_n^2) ^{\frac n2+1}}|u|^{\alpha} dm_{n-1}(u)
\\& =  \frac {\Gamma (\frac n2)}{\pi ^{\frac n2}} x_n^{\alpha-1}  \int_{\mathbb {R}^{n-1}} 
\frac{|v |^2 + n-1}{ ( | v |^2 + 1) ^{\frac n2+1}} |v|^{\alpha} dm_{n-1}(v),
\end{split}\end{equation*}
where  $v = {x_n} ^{-1} u$. Let us denote by $J_n$ the last integral  expression along  with the multiplicative  factor.

Introducing the polar coordinates,  it is straightforward  to calculate the exact value of $J_n$ and to obtain  the estimate
\begin{equation*}
J _n \le 2\frac{ \Gamma (\frac 12 -\frac { \alpha}2) }{ \Gamma(\frac12) }
\frac{\Gamma(\frac{\alpha}{2}+ \frac  {   n-1}2 )}{  \Gamma( \frac {n-1}2) }x_n^{\alpha-1}.
\end{equation*}

In case  $n\ge 3$ we can  use  the  Gautschi   inequality for $x =\frac \alpha 2+ \frac { n- 3}2 $ and $s=1-\frac \alpha2$,                 and some elementary inequalities, in order to obtain
\begin{equation*}
\frac{\Gamma(\frac\alpha 2  + \frac{ n-1}2 ) }{  \Gamma( \frac {n-1}2) } 
= \frac{   \Gamma(x+1  ) }{ \Gamma( x + s) } \le(n-1)^{\frac12}.
\end{equation*}
This estimate along with $\Gamma\left(\frac 12 -\frac {\alpha}2\right)\le \frac {2}{1-\alpha}$,  and some trivial ones, lead to the final  estimate
\begin{equation*}
J_n\le \frac {4(n-1)^{\frac12}}{1-\alpha}x_n^{\alpha-1}.
\end{equation*}
(The last inequality is also valid in the case $n=2$, as may be checked immediately.)                        It follows the estimate for the $n$-th  partial derivative stated in   our  theorem.
\end{proof}

\begin{proof}[Proof  of Theorem  \ref{TH.MAIN.2}]
Let $U$ be a bounded real-valued harmonic function on $\mathbb{R}^n_+$, continuous on the closure of this domain, and let it satisfies
\begin{equation}\label{EQ.VERTICAL}
|U(x', x_n) - U (x', y_n)| \le   C   y_n^\alpha,\quad  x'\in \mathbb{R}^{n-1},\, 0< x_n < y_n,
\end{equation}
where $C$ is a constant.

For every    $\lambda \in (0,\frac 1\alpha)$   introduce   the following harmonic    function on the upper   half-space
\begin{equation*}
V_\lambda(x', x_n ) =  U (x', x_n + \lambda),\quad x  = (x',x_n) \in \mathbb {R}^n_+,
\end{equation*}
and attach to it   the number
\begin{equation}\label{EQ.ALAMBDA}
A_\lambda  = \sup _{  (x',x_n) \in \mathbb {R}^n_+  } x_n^{1 - \alpha}  \left| \frac {\partial} {\partial x_n} V_\lambda(x',x_n) \right|.
\end{equation}
We will firstly show that $A_\lambda$ is finite for every $\lambda\in(0,1)$. Then we will proceed to obtain an upper bound for $A_\lambda$, $\lambda>0$,   which  does  not depend on  $\lambda$.

Due to the assumption that $U$  is bounded, we may apply the  Kresin-Maz'ya estimate given in  Proposition \ref{PROP.KHAVINSON}. It follows
\begin{equation*}\begin{split}
x_n^{1 - \alpha}  \left| \frac {\partial} {\partial x_n}V_\lambda (x',x_n) \right|  & = {x_n^{1 - \alpha}}
\left| \frac {\partial} {\partial x_n} U(x',x_n+\lambda) \right| 
\\& \le \frac {x_n^{1 - \alpha}}{x_n+\lambda} M_n \sup _{y\in \mathbb{R}^n_+}|U(y)|
\\& \le \frac {M_n}{\lambda^\alpha} \sup _{y\in \mathbb{R}^n_+}|U(y)|
\end{split}\end{equation*}
for every $(x',x_n) \in \mathbb {R}^n_+$, since the maximum of the function  $x\to \frac{x^{1-\alpha}}{x+\lambda}$,  $x\in(0,\infty)$, achieves  at  $x = \frac 1\alpha  -  \lambda$.  From  the
last displayed  expression we  conclude that $A_\lambda$ is finite.

Applying    the Taylor expansion  on one variable function $s\to V_\lambda(x',s)$, where  $x=(x',x_n)\in\mathbb{R}^{n}$ is fixed, around  $x_n\in (0,\infty)$, we obtain that for $s\in (0,x_n)$
there exists  $t\in (s,x_n)$ such that
\begin{equation*}\begin{split}
V_\lambda (x',s) &  =  V_\lambda  (x', x_n) + \frac {\partial } {\partial x_n}V_\lambda (x',x_n)  (s- x_n) 
\\&+ \frac {\partial^2 } {\partial x_n^2} V_\lambda(x',t) \frac {(s-x_n)^2}2.
\end{split}\end{equation*}
In the above expansion let us  take   $s = (1 - \gamma)  x_n$, where  $\gamma \in (0,1)$  will be chosen letter. This equation takes the form
\begin{equation*}\begin{split}
V_\lambda (x', x_n - \gamma  x_n) & =  V_\lambda  (x', x_n) - \frac {\partial } {\partial x_n} V_\lambda (x',x_n) \gamma x_n
\\&+ \frac {\partial^2 } {\partial x_n^2}V_\lambda (x',t) \frac {\gamma^2 x_n^2}2.
\end{split}\end{equation*}
From the last  equality we have
\begin{equation*}\begin{split}
\frac {\partial } {\partial x_n}V_\lambda(x',x_n) \gamma x_n  & = V_\lambda  (x', x_n) - V_\lambda (x', x_n - \gamma  x_n)
\\&+ \frac {\partial^2 } {\partial x_n^2}V_\lambda (x',t) \frac {\gamma^2 x_n^2}2.
\end{split}\end{equation*}
Taking the absolute  values on both sides above,  then  using  the   triangle inequality,  we derive
\begin{equation*}\begin{split}
\left| \frac {\partial } {\partial x_n} V_\lambda (x',x_n)\right |\gamma x_n & \le | V_\lambda  (x', x_n) - V_\lambda (x', x_n - \gamma  x_n) |
\\&+ \left|\frac {\partial^2} {\partial x_n^2} V_\lambda(x',t) \right|  \frac { \gamma^2 x_n^2}2.
\end{split}\end{equation*}

Applying now the assumed  inequality  \eqref{EQ.VERTICAL} we  obtain
\begin{equation*}
\left| \frac {\partial} {\partial x_n} V_\lambda(x',x_n) \right|\gamma x_n \le C x_n ^{\alpha} +
\sup_{y\in {B}(x, \gamma x_n)} \left|\frac {\partial^2} {\partial x_n^2}V_\lambda (y) \right|\frac { \gamma^2 x_n^2}2,
\end{equation*}
since $t\in (x_n - \gamma x_n,x_n)$.

By  Lemma \ref{LE.KN}  we have the estimate
\begin{equation*}
\left|\frac {\partial^2} {\partial x_n^2}V_\lambda(y)\right|
\le \frac {K_n}{\frac {1-\gamma}2 x_n} \sup_{z\in {B}(x,\frac{1+\gamma}2 x_n)}
\left |\frac {\partial} {\partial x_n} V_\lambda(z)\right|,
\end{equation*}
where we have applied the lemma for the harmonic function $\frac {\partial V_\lambda}{\partial x_n}$ on the domain $D = {B}(x,\frac{1+\gamma}2 x_n)$; notice  that for $y\in {B}(x, \gamma x_n)$
we have  $d(y,\partial D)\ge \frac {1-\gamma}2 x_n$.

We have   arrived   to  the following inequality
\begin{equation*}
\left| \frac {\partial } {\partial x_n}V_\lambda (x',x_n) \right|\gamma x_n \le C x_n ^{\alpha} +
\frac {K_n}{\frac {1-\gamma}2 x_n} \sup_{z\in {B}(x,\frac{1+\gamma}2 x_n)}
\left |\frac {\partial} {\partial x_n}V_\lambda(z)\right|\frac { \gamma^2 x_n^2}2,
\end{equation*}
which after   dividing  by $\gamma x_n$  transforms to
\begin{equation*}
\left| \frac {\partial } {\partial x_n} V_\lambda (x',x_n)\right| \le \frac{C}\gamma x_n ^{\alpha-1}
+ \gamma \frac {K_n}{{1-\gamma}} \sup_{z\in {B}(x,\frac{1+\gamma}2 x_n)}\left |\frac {\partial} {\partial x_n} V_\lambda(z)\right|.
\end{equation*}

We will now estimate the second therm above via $A_\lambda$.  Note  that for $z = (z', z_n)\in {B}(x,\frac{1+\gamma}2 x_n)$  we have  $z_n \ge x_n-\frac{1+\gamma}2 x_n= \frac {1-\gamma}2 x_n$.
Having on mind the  relation  \eqref{EQ.ALAMBDA}, we obtain
\begin{equation*}
\left| \frac {\partial} {\partial x_n} V_\lambda(z',z_n) \right|\le A_\lambda z_n^{ \alpha-1}\le A_\lambda \left( \frac {1-\gamma}2x_n\right)^{ \alpha-1}.
\end{equation*}
Therefore,
\begin{equation*}
\sup_{z\in {B}(x,\frac{1+\gamma}2 x_n)}\left |\frac {\partial} {\partial x_n}V_\lambda(z)\right|\le A_\lambda \left( \frac {1-\gamma}2\right)^{-1}x_n^{ \alpha-1}.
\end{equation*}

Going back to the main inequality,  and dividing it by $x_n^{\alpha -1}$, we arrive to the following one
\begin{equation*}
x_n^{1-\alpha}\left| \frac {\partial  } {\partial x_n} V_\lambda (x', x_n)\right|
\le \frac{C}\gamma + \frac {\gamma}{1-\gamma} K_n A_\lambda \left(\frac {1-\gamma}2\right)^{- 1}.
\end{equation*}
If we  take supremum on the left side with  respect to  $x = (x', x_n)\in \mathbb{R}^n_+$,  we obtain
\begin{equation*}
A_\lambda\le \frac{C}\gamma + \frac {2\gamma K_n}{(1-\gamma)^2}  A_\lambda.
\end{equation*}

Let us take  for  $\gamma \in (0,1)$ the  number  such that
\begin{equation*}
\frac {2\gamma K_n}{(1-\gamma)^2}   = \frac 12.
\end{equation*}
It ie easy to check that the unique solution to this equation  in $(0,1)$ is given by
\begin{equation*}
\gamma =  K_n((1+ {K_n}^{-1} )^{\frac 12}-1)^2.
\end{equation*}

From the last inequality concerning $A_\lambda$ it follows
\begin{equation*}
A_\lambda \le \frac{2 C}{\gamma} \Leftrightarrow A_\lambda \le \frac {2 C}{  K_n((1+ {K_n}^{-1} )^{\frac 12}-1)^2}.
\end{equation*}
Using the elementary inequality
\begin{equation*}
\frac{2}{x((1+ x^{-1})^{\frac 12} -1)^2}\le 13 x,\quad  x>1,
\end{equation*}
and then  $K_n\le (n+2)^{\frac 12}$,  we obtain
\begin{equation*}
A_\lambda\le 13 (n+2)^{\frac 12} C,
\end{equation*}
where the right side does not depend on $\lambda$.

We have just  proved the inequality
\begin{equation*}
\left| \frac {\partial} {\partial x_n} U(x',x_n+\lambda) \right|  \le 13 (n+2)^{\frac 12} C x_n^{ \alpha-1},\quad x=(x',x_n)\in \mathbb{R}^n_+
\end{equation*}
for every $\lambda \in (0,1)$.  If we finally left $\lambda\to 0$ above, we obtain  the estimate  of the  $n$-th  partial derivative for  $U$  which we aimed  to prove.

It remains to apply Lemma \ref{LE.GRAD.ALPHA}  and  Lemma \ref{LE.GM}  in order to obtain the other  two  estimates in the theorem.
\end{proof}

\begin{proof}[Proof of Theorem \ref{TH.MAIN}]
Let us introduce one more semi-norm
\begin{equation*}
\|U\|_4 = \sup_{x'\in \mathbb{R}^{n-1},\,  0<x_n<y_n} \frac { |  U (x',x_n) - U  (x',y_n) |  } {y_n^\alpha}.
\end{equation*}

It is enough to show this theorem in the case of real-valued bounded harmonic functions. Let $U$ be a such one. We have  $\| U\|_2 \le \|U\|_1$,  and by Theorem \ref{TH.MAIN.1} there  exists a
constant  $C_2$ such that $\| U\|_1 \le C_2 \|U\|_2$. Applying Theorem  \ref{TH.MAIN.2} we obtain a constant $C_4$ such that  $\|U\|_1\le C_4 \|U\|_4$. Since   $\|U\|_4 \le \|U\|_3\le \|U\|_1$,
it follows the statement of  this  theorem.
\end{proof}

\begin{proof}[Proof of Theorem \ref{TH.MAIN.4}]
This theorem is  a consequence of Theorem \ref{TH.MAIN.1}, Theorem  \ref{TH.MAIN.2},  and some elementary consideration.  We follow the approach  as in
\cite{DYAKONOV.MZ.2005, DYAKONOV.MZ.2006, PAVLOVIC.RMI}.

It enough to prove that the condition (H2) implies (H1), and (H3) implies (H1).

Note that the condition (H2) implies the following one
\begin{equation*}
| U  (x',0) -  U  (y',0)| \le 2 C_2  |x'- y'|^\alpha,\quad  x ',y' \in \mathbb {R}^{n-1}.
\end{equation*}
Indeed, if $U (x',0) U (y',0)\ge 0$, then we have
\begin{equation*}
|U(x', 0) - U (y', 0)| = | |U(x', 0) | - |U (y', 0)|| \le C_2 |x' - y'|^\alpha.
\end{equation*}
On the other hand, if $U (x', 0) U (y', 0)< 0$,   because the  function $U(x',0)$ is continuous on $\mathbb {R}^{n-1}$, one can find $z'\in [x', y']$ ($[x',y']\subseteq \mathbb {R}^{n-1}$ is a
segment with endpoints at  $x'$ and $y'$) such that  $U (z', 0) = 0$. Having on mind the preceding case, we obtain
\begin{equation*}\begin{split}
|U(x',0) - U (y',0)| & \le  |U(x',0) - U (z',0)|   +  |U (z',0)  - U (y',0)|\\&\le  C_2 |x'- z' |^\alpha + C_2 |z '-y'|^\alpha\le 2 C_2  |x'- y'|^\alpha.
\end{split}\end{equation*}
Therefore, $U(x', ,0)\in \Lambda^\alpha (\mathbb {R}^{n-1})$, so it  remains to apply Theorem \ref{TH.MAIN.1}.

Assume now that the condition (H3) holds. We will derive the following inequality
\begin{equation*}
|U(x', x_n) - U (x', y_n)| \le 4C_3 y_n^\alpha,\quad  x' \in \mathbb {R}^{n-1}, 0< x_n<y_n.
\end{equation*}
As before, we  will separate the proof in  two cases.

If $U(x', x_n)  U (x', y_n) \ge  0$, then we have
\begin{equation*}\begin{split}
|U(x', x_n ) - U (x', y_n)|     &   = | |U(x', x_n) | - | U (x', y_n) ||
\\&  \le| |U(x', x_ n) | - | U (x',0) || +| |U(x',0) | - | U (x',y_n) ||
\\&\le C_3  x_n^\alpha +  C_3 y_n^\alpha\le 2 C_3 y_n^\alpha.
\end{split}\end{equation*}

If $U(x',x_n) U (x',y_n) < 0$, then because  of the continuity of $U(x', t)$, $t\ge 0$,  there exists $z_n\in (x_n,y_n)$ such that $U(x', z_n)= 0$.         Applying  the first  case, we obtain
\begin{equation*}\begin{split}
|U(x', x_n) - U (x', y_n)| & \le    | U(x', x_n)  -  U (x', z_n) | + | U(x',z_n)  -   U (x', y_n)|
\le 4C_3 y^\alpha_n.
\end{split}\end{equation*}

By   Theorem \ref{TH.MAIN.2}  we  conclude that   (H3)  implies  (H1).
\end{proof}


\begin{thebibliography}{10}

\bibitem{AXLER.BOOK}
Sh. Axler, P.  Bourdon,  W. Ramey,
\textit{Harmonic function theory},
Springer-Verlag,  New York, 2001.

\bibitem{DYAKONOV.ACTM}
K.M. Dyakonov,
\textit{Equivalent norms on Lipschitz-type spaces of holomorphic functions},
Acta Mathematica  \textbf{178} (1997), 143--167.

\bibitem{DYAKONOV.AIM}
K.M. Dyakonov,
\textit{Holomorphic functions and quasiconformal mappings with smooth moduli},
Advances in Mathematics  \textbf{187} (2004), 146--172.

\bibitem{DYAKONOV.MZ.2005}
K.M. Dyakonov, \textit{Strong Hardy-Littlewood theorems for analytic functions and mappings of finite distortion},
Math. Z. \textbf{249} (2005), 597--611.

\bibitem{DYAKONOV.MZ.2006}
K.M. Dyakonov,
\textit{Addendum to ''Strong Hardy-Littlewood theorems for analytic functions and mappings of finite distortion''},
Math. Z. \textbf{254} (2006), 433--437.

\bibitem{GEHRING.AASF}
F.W. Gehring and  O. Martio,
\textit{Lipschitz classes and quasiconformla mappings},
Annales Academi{\ae} Scientiarum Fennic{\ae} \textbf{10} (1985), 203--219.

\bibitem{LAPP.AASF}
V. Lappalainen,
\textit{Liph-extension domains},
Annales Academi{\ae} Scientiarum Fennic{\ae} Dissertationes, 1985.

\bibitem{LIU.MA}
C. Liu,
\textit{A proof of the Khavinson conjecture},
Mathematische Annalen \textbf{380} (2021), 719--732.

\bibitem{KRESIN.DCDS}
G. Kresin,  V. Maz’ya,
\textit{Optimal estimates for the gradient of harmonic functions in the multidimensional half-space},
Discrete and Continuous Dynamical Systems \textbf{28} (2010), 425--440.

\bibitem{PAVLOVIC.ACTM}
M. Pavlovi\'{c},
\textit{On Dyakonov's paper Equivalent norms on Lipschitz-type spaces of holomorphic functions},
Acta Mathematica  \textbf{183} (1999), 141--143.

\bibitem{PAVLOVIC.RMI}
M. Pavlovi\'{c},
\textit{Lipschitz conditions on the modulus of a harmonic function},
Revista Matematica Iberoamericana \textbf{23} (2007), 831--845.

\bibitem{PAVLOVIC.BOOK}
M. Pavlovi\'{c},
\textit{Function classes on the unit disc},
De Gruyter Studies in Mathematics, Berlin, 2019.

\bibitem{RAVISANKAR.PHD}
S. Ravisankar,
\textit{Lipschitz properties of harmonic and holomorphic functions},
Ph.D. dissertation, The Ohio State University, 2011.

\bibitem{RAVISANKAR.CVEE}
S. Ravisankar,
\textit{Transversally Lipschitz harmonic functions are Lipschitz},
Complex Variables and Elliptic Equations \textbf{58} (2013), 1685--1700.

\bibitem{STEIN.BOOK}
E.M. Stein,
\textit{Singular integrals},
Princeton University Press, New Jersey,  1970.

\end{thebibliography}
\end{document}